\documentclass[12pt]{amsart}
\usepackage{amscd, amsfonts, amssymb}
\usepackage{fullpage}
\usepackage{latexsym}
\usepackage{verbatim}
\usepackage{enumerate}
\usepackage[all,cmtip]{xy}
\usepackage{color}
\usepackage{comment}

\numberwithin{equation}{section}

\newtheorem{thm}[equation]{Theorem}

\newtheorem{lem}[equation]{Lemma}

\theoremstyle{definition}

\theoremstyle{remark}
\newtheorem{rem}[equation]{Remark}
\theoremstyle{remark}

\DeclareMathOperator{\Hom}{Hom}
\DeclareMathOperator{\Lie}{Lie}
\DeclareMathOperator{\Ad}{Ad}
\DeclareMathOperator{\ev}{ev}

\newcommand{\deq}{\overset{\textup{def}}{=}}
\newcommand{\F}{\mathbb{F}}
\newcommand{\N}{\mathbb{N}}
\newcommand{\C}{\mathbf{C}}
\newcommand{\D}{\mathbf{D}}


\date{\today}

\title[Rigid representations of triangle groups]
{Rigid representations of triangle groups}

\author[A.J.\ Litterick]{Alastair Litterick}
\address%
{A.J.\ Litterick,
Fakult\"{a}t f\"{u}r Mathematik,
Ruhr-Universit\"{a}t Bochum,
Universit\"{a}tsstra{\ss}e 150,
D-44780 Bochum,
Germany
\and
Fakult\"{a}t f\"{u}r Mathematik,
Universit\"{a}t Bielefeld,
Postfach 100131,
D-33501 Bielefeld,
Germany}
\email{ajlitterick@gmail.com}
\thanks{The first author is supported by the Alexander von Humboldt Foundation.}

\author[B.\ Martin]{Benjamin Martin}
\address
{Department of Mathematics,
University of Aberdeen,
King's College,
Fraser Noble Building,
Aberdeen AB24 3UE,
United Kingdom}
\email{B.Martin@abdn.ac.uk}

\begin{document}

\begin{abstract}
We prove a generalization of a conjecture of C.\ Marion on generation properties of finite groups of Lie type, by considering geometric properties of an appropriate representation variety and associated tangent spaces.
\end{abstract}

\maketitle

\section{Introduction}

Let $G$ be a linear algebraic group over an algebraically closed field $k$ of characteristic $p\geq 0$.  Given $d\in \N$, let $j_d$ be the largest dimension of a conjugacy class of elements of $G$ of order $d$.  For $a,b$ and $c$ positive integers, let $T(a,b,c)$ be the triangle group given by $T(a,b,c)= \langle x,y,z\mid x^a= y^b= z^c=xyz= 1\rangle$. In this note we prove a conjecture of Marion (cf.\ \cite{Marion2010}).

\begin{thm}
\label{thm:triangle}
 Suppose that $G$ is simple and defined over a finite field $\F_{q}$, and suppose $j_a + j_b + j_c = 2\dim(G)$.  Then there are only finitely many $r\in \N$ such that the finite group of Lie type $G(q^r)$ is a quotient of $T(a,b,c)$.
\end{thm}

In fact our results are more general. For integers $a_1,a_2,\ldots,a_n \in \N$ let
\[ T(a_1,a_2,\ldots,a_n) \deq \left< x_1,\ldots,x_n \mid x_1^{a_1} = \ldots = x_{n}^{a_n} = x_1 x_2 \ldots x_n = 1 \right>. \]
Recall also that if an algebraic group $G$ is defined over a finite field $\F_q$, then a \emph{Steinberg endomorphism} of $G$ is a homomorphism $F \colon G \to G$ of algebraic groups such that some power of $F$ is the Frobenius endomorphism corresponding to an $\F_{q^t}$-structure of $G$ $(t \ge 1)$. In particular, the fixed-point subgroup $G^{F}$ is finite whenever $F$ is a Steinberg endomorphism. Let $\mathcal{S}_{G}$ be the set of isomorphism classes of groups of the form $G^{F}$, as $F$ varies over Steinberg endomorphisms of $G$.


\begin{thm} \label{thm:triangle_and_more}
Suppose that $G$ is connected and reductive, and let $a_1,a_2,\ldots,a_n \in \N$ such that $\sum_{i = 1}^{n} j_{a_i} = 2 (\dim G - \dim Z(G))$. Then only finitely many members of $\mathcal{S}_{G}$ are a quotient of the group $T(a_1,\ldots,a_n)$.
\end{thm}

\begin{rem}
Theorem~\ref{thm:triangle} was proved `generically' in \cite[Theorem~1.7]{Larsen2014}, under the hypothesis that the defining characteristic of $G$ is coprime to the integers $a$, $b$, $c$, and to the determinant of the Cartan matrix of $G$. Using this, the authors show that a certain cohomology group vanishes, which implies that a certain map of tangent spaces is surjective. Our results follow Theorem~\ref{thm:general} below, which is in a similar vein but with weaker hypotheses. Indeed, one can show that under the restriction on the characteristic given in \cite{Larsen2014}, each appropriate generating set for a finite group $H = G(q^r)$ will satisfy the hypothesis of Theorem~\ref{thm:general} below.
\end{rem}

\section{Representation varieties}

We start with a few preliminaries. For a linear algebraic group $G$ we write $\mathfrak{g}$ for the Lie algebra $\Lie(G)$ considered as the adjoint $G$-module, and $g\cdot X$ for $\Ad_g(X)$ if $g\in G$ and $X\in \mathfrak{g}$. Given $H \leq G$, we set $D_H(\mathfrak{g})= \{X- h\cdot X\mid h\in H\}$ and let $\mathfrak{g}_H = \mathfrak{g}/D_H(\mathfrak{g})$, the space of coinvariants of $H$ on $\mathfrak{g}$. The Zariski closure of $H$ is denoted $\overline{H}$.

Let $\Gamma$ be a finitely generated group. Define $\Hom(\Gamma, G)$ to be the set of homomorphisms from $\Gamma$ to $G$.  It is well-known that $\Hom(\Gamma, G)$ is an affine algebraic variety over $k$, called the \emph{representation variety}.  To see this concretely, let $\gamma_1,\ldots, \gamma_n$ be a set of generators for $\Gamma$. Then the map $\rho \mapsto (\rho(\gamma_1),\ldots,\rho(\gamma_n))$ embeds $\Hom(\Gamma, G)$ into $G^{n}$ as the closed subvariety of $n$-tuples satisfying the defining relations of $\Gamma$.  In particular, if $F_n$ denotes the free group on $n$ generators then we may (and shall) identify $\Hom(F_n, G)$ with $G^n$ via a choice of free generators.  If $\gamma\in \Gamma$ then we have a morphism of varieties $\ev_\gamma\colon \Hom(\Gamma, G)\to G$ given by $\ev_\gamma(\rho)= \rho(\gamma)$.

If $C_1,\ldots, C_n$ are conjugacy classes of $G$, let $\C = C_1 \times \cdots \times C_n$. Define $f\colon \C \to G$ by $f(c_1,\ldots, c_n)= c_1\ldots c_n$ and set $\C_0= \{(c_1,\ldots, c_n)\in \C \mid c_1\ldots c_n= 1\}= f^{-1}(1)$. The group $G$ acts on $G^{n}$ by conjugation, and this action stabilises both $\Hom(\Gamma, G)$ and $\C_0$.

An element $\rho = (g_1\ldots,g_n) \in \Hom(\Gamma, G)$ determines an $n$-tuple of classes $\C = C_1 \times \cdots \times C_n$ (where $g_i \in C_i$ for each $i$), and we may regard $\C$ as a (quasi-affine) subvariety of $G^n$. If $\rho \in \C_0$ then we say that $\rho$ is \emph{rigid} if the orbit $G\cdot \rho$ is an open subset of $\C_0$. It is immediate that $G\cdot \rho$ is open if and only if $G^{\circ}\cdot \rho$ is open. Note that if the generators of $\Gamma$ multiply to $1$, then $\rho \in \C_0$ for all $\rho$.


Recall that a subgroup $H$ of a connected reductive group $G$ is said to be \emph{$G$-irreducible} ($G$-ir) if $H$ is not contained in any proper parabolic subgroup of $G$; in this case, if $H$ is closed then $H^{\circ}$ is reductive and the centraliser $C_G(H)$ is finite. For a subset $X \le \Hom(\Gamma, G)$, write $X_{\rm ir}$ for the set of homomorphisms $\rho\in X$ such that $\overline{\rho(\Gamma)}$ is $G$-ir.

\begin{thm}
\label{thm:general}
Suppose that $G$ is semisimple and that there exists $\rho = (c_1,\ldots,c_n) \in (\C_{0})_{\rm ir}$, such that the closed subgroup $H= \overline{\langle c_1,\ldots, c_n\rangle}$ satisfies $\mathfrak{g}_H = 0$. Then the following are equivalent.\smallskip\\
 (a) $\rho$ is rigid.\smallskip\\
 (b) $\dim(C_1)+\cdots + \dim(C_n)= 2\dim(G)$.\smallskip\\
 (c) $\dim(C_1)+\cdots + \dim(C_n)\leq 2\dim(G)$.
\end{thm}


We postpone the proof of Theorem~\ref{thm:general} to the next section, and now use it to prove Theorem~\ref{thm:triangle_and_more} and hence Marion's conjecture (Theorem~\ref{thm:triangle}) as follows. We first reduce to the case that $Z(G)$ is finite, i.e.\ $G$ is semisimple. We then replace $G$ by its simply-connected cover. Neither of these operations changes the quantity $\dim G - \dim Z(G)$, nor the dimension of the conjugacy class of a fixed element. In order to make use of these reductions, we prove a slightly stronger result than Theorem~\ref{thm:triangle_and_more} for simply-connected semisimple groups.

For our first lemma, note that if $H$ is a closed, characteristic subgroup of a connected reductive group $G$, then a Frobenius morphism of $G$ induces a Frobenius morphism of $G/H$, and if $H$ is connected then $G^{F}/H^{F} \to (G/H)^{F}$ is an isomorphism \cite[Proposition 23.2]{Malle2011}.

\begin{lem}
Let $G$ be connected and reductive, let $Z = Z(G)^{\circ}$ be the identity component of the centre of $G$, and write $\underline{G} = G/Z$. If infinitely many distinct members of $\mathcal{S}_{G}$ are images of some $T(a_1,\ldots,a_n)$, then so are infinitely many members of $\mathcal{S}_{\underline{G}}$.
\end{lem}

\proof We claim that if $F$ is a Frobenius morphism, then the order of $Z^{F}$ is bounded in terms of $G$ and $(a_1,\ldots,a_n)$ but independently of $F$; this implies that an infinite sequence of distinct groups of the form $G^{F}$ gives rise to an infinite sequence of distinct groups of the form $\underline{G}^{F}$.

The group $[G,G] \cap Z \le Z([G,G])$ has finite order bounded by the Lie type of $[G,G]$. The quotient by this subgroup sends $Z$ onto a connected central subgroup of $G/([G,G] \cap Z)$, whose fixed points under the induced Frobenius morphism contain the image of $Z^{F}$. Thus it suffices to prove the claim assuming that $[G,G] \cap Z = 1$, so that $G = [G,G] \times Z$. But in this case, if $G^{F}$ is a quotient of $T(a_1,\ldots,a_n)$ then so is $Z^{F}$, and $|Z^{F}|$ is then bounded by the size of the abelianization of $T(a_1,\ldots,a_n)$, which is bounded in turn by $a_1 a_2 \ldots a_n$. \qed

Thus in the setting of Theorem~\ref{thm:triangle_and_more} it suffices to prove that only finitely many members of $\mathcal{S}_{\underline{G}}$ are images of $T(a_1,\ldots,a_n)$. Note that the dimension of a conjugacy class of $G$ is equal to the dimension of its image in $\underline{G}$.

Next, let $\widetilde{G}$ be the simply-connected cover of $\underline{G}$, with canonical projection $\pi\colon \widetilde{G}\to \underline{G}$. Then $\pi$ gives rise to an obvious map $\widetilde{G}^{n} \to \underline{G}^{n}$ which we also denote by $\pi$, abusing notation. We identify $\Hom(T(a_1,a_2,\ldots,a_n), \underline{G})$ with the set of $n$-tuples $\mathbf{h} = (h_1,h_2,\ldots,h_n) \in \underline{G}^{n}$ such that $h_i^{a_i} = 1$ for each $i$ and $h_1 h_2 \ldots h_n = 1$. Then $\pi^{-1}(\Hom(T(a_1,a_2,\ldots,a_n), \underline{G}))$ is a set of $n$-tuples in $\widetilde{G}^{n}$ of elements whose product lies in $Z(\widetilde{G})$. We identify these with $(n+1)$-tuples $(\widetilde{h_{1}},\widetilde{h_{2}},\ldots,\widetilde{h_n},\widetilde{h_{n+1}}) \in \widetilde{G}^{n+1}$ such that $\widetilde{h_{1}}\widetilde{h_{2}} \ldots \widetilde{h_n}\widetilde{h_{n+1}} = 1$ and $\widetilde{h_{n+1}} \in Z(\widetilde{G})$. Since $Z(\widetilde{G})$ is finite, the set of such $(n+1)$-tuples occurring this way is contained in a finite union of subsets of the form $\Hom(T(\widetilde{a_1},\widetilde{a_1},\ldots,\widetilde{a_n},\widetilde{a_{n+1}}), \widetilde{G})$ for integers $\widetilde{a_{i}}$ such that $\widetilde{a_{n+1}} \le |Z(\widetilde{G})|$ and $a_{i} \le \widetilde{a_i} \le |Z(\widetilde{G})|a_{i}$ for $1 \le i \le n$. Moreover the dimension of the conjugacy class $\widetilde{G} \cdot h_{i}$ equals that of $\underline{G} \cdot h_{i}$ for $1 \le i \le n$, and is $0$ for $i = n + 1$.

The Steinberg endomorphism $F$ is induced from a Steinberg endomorphism of $\widetilde{G}$ (cf.\ \cite[Proposition~22.7]{Malle2011}), and $\pi^{-1}(\underline{G}^{F})$ contains (but may be strictly larger than) the subgroup $\widetilde{G}^{F}$. The image $\pi(\widetilde{G}^{F})$ has index in $\underline{G}^{F}$ equal to $Z(\widetilde{G}^{F})$ (see \cite[\S 24.1]{Malle2011}), in particular this index is bounded by $|Z(\widetilde{G})|$, independently of $F$.

So suppose infinitely many members of $\mathcal{S}_{G}$ are images of $T(a_1,\ldots,a_n)$. Taking the image in $\underline{G}$ and then lifting to $\widetilde{G}$, by the above discussion we obtain infinitely many distinct subgroups of $\widetilde{G}$ which are generated by an $(n+1)$-tuple of elements in $\Hom(T(\widetilde{a_1},\ldots,\widetilde{a_n},\widetilde{a_{n+1}}),\widetilde{G}) \cap (C_1 \times \ldots \times C_n,C_{n+1})$, where $C_{i}$ are conjugacy classes of $\widetilde{G}$ with $\sum_{i = 1}^{n+1} \dim C_{i} \le \sum_{i = 1}^{n} j_{a_i}$. Each of the subgroups occurring this way contains a subgroup of the form $\widetilde{G}^{F}$, with index bounded independently of $F$.

It therefore suffices show the following.

\begin{lem}
Let $G$ be a simply-connected semisimple group and let $\mathcal{S}$ be a collection of subgroups of $G$, such that each member of $\mathcal{S}$ contains a subgroup of the form $G^{F}$ for a Steinberg endomorphism $F$ of $G$, and such that for any infinite sequence of pairwise distinct members of $\mathcal{S}$, the corresponding sequence of subgroups $G^{F}$ has unbounded order.

If $C_1,\ldots,C_n$ are conjugacy classes of $G$ such that $\sum \dim C_{i} \le 2 \dim G$, then only finitely many members of $\mathcal{S}$ are generated by elements in an $n$-tuple in the set
\[ \Hom(T(a_1,\ldots,a_n),G) \cap (C_1 \times \ldots \times C_n). \]
\end{lem}

\proof We claim that all but finitely many subgroups $H \in \mathcal{S}$ are $G$-irreducible and satisfy $\mathfrak{g}_{H} = 0$, which allows us to apply Theorem~\ref{thm:general}. By the hypothesis on $\mathcal{S}$, it suffices to prove this for the subgroups $G^{F}$ themselves, for if $G^{F} \le H$ and $G^{F}$ is $G$-irreducible and has no coinvariants on $\mathfrak{g}$, then this also holds for $H$.

The hypothesis on $G$ implies that $G$ is the direct product of its simple factors, and its Lie algebra $\mathfrak{g}$ is the direct sum of the Lie algebras of these factors. Moreover the Lie algebra of a simply-connected simple group has no coinvariants (see for instance \cite[Table 1]{Hogeweij1982}). Thus $\mathfrak{g}_{G} = 0$. Next, let $F$ be a Steinberg endomorphism of $G$, and consider the image of $G^{F}$ under projection to some direct factor $G_{1}$ of $G$. From the classification of Steinberg endomorphisms of semisimple groups (cf.\ \cite[Corollary 1.5.16]{Geck2016}) it follows that this image is a finite group of Lie type, in fact it is equal to $G_1^{F^{m}}$ for some $m \ge 1$ such that $F^{m}(G_1) = G_1$. Thus if $\mathfrak{g}_{G^F} \neq 0$ for infinitely many pairwise non-isomorphic subgroups $G^{F}$, then there exists a simple factor of $G$ (say $G_1$ without loss of generality) and an infinite collection of subgroups $G_1^{F_1}$, $G_1^{F_2}$, $\ldots$ for Steinberg endomorphisms $F_i$ of $G_1$, such that $G_{1}^{F_i}$ has coinvariants on the Lie algebra $\mathfrak{g}_{1}$ of $G_1$ for all $i$. On the other hand, for all but finitely many $F_{i}$, the $G_{1}$-module composition factors of $\mathfrak{g}_{1}$ all remain irreducible upon restriction to $G_{1}^{F_{i}}$ \cite[Theorem~2.11]{Humphreys2005}, with non-isomorphic factors remaining non-isomorphic. Then results relating cohomology of $G_{1}^{F_{i}}$ to that of $G_{1}$ \cite[Theorem~7.1]{Cline1977}, combined with results relating group cohomology to submodule structure \cite[Proposition~1.4]{Liebeck1998}, tell us that for all but finitely many values of $i$, each $G_{1}^{F_{i}}$-submodule of $\mathfrak{g}_{1}$ is a $G_{1}$-module. In particular, this implies that $G$ has coinvariants on its Lie algebra; a contradiction.

Note also that, again with finitely many exceptions, the group $G_{1}^{F_{i}}$ is quasi-simple and does not map non-trivially into a proper Levi subgroup of $G_1$; it follows that with finitely many exceptions, a group of Lie type $G^{F}$ is $G$-irreducible.

We have therefore shown that, with finitely many exceptions, the subgroups $H \in \mathcal{S}$ are $G$-irreducible and satisfy $\mathfrak{g}_{H} = 0$. Let $\mathcal{S}'$ denote the subset of $\mathcal{S}$ satisfying these conditions. If $H \in \mathcal{S}'$ and $\rho = (h_1,\ldots,h_n) \in \Hom(T(a_1,a_2,\ldots,a_n),G)$ such that $\rho(T(a_1,a_2,\ldots,a_n)) = H$, we may regard $\rho$ as an element of $(\C_0)_{\rm ir}$ where we choose $C_i= G\cdot h_i$ for $i = 1,\ldots,n$. By Theorem~\ref{thm:general}, $\rho$ is rigid. Hence if $H_1$, $H_2 \in \mathcal{S}'$ and $\rho_1$, $\rho_2 \in \Hom(T(a_1,a_2,\ldots,a_n), G)$ have respective images $H_1$ and $H_2$, then $G \cdot \rho_1$ and $G\cdot \rho_2$ are nonempty disjoint open subsets of $\C_0$, hence cannot both be contained in the same irreducible component. But the variety $\C_0$ has only finitely many irreducible components, and the desired result follows. \qed


\section{Tangent spaces and 1-cocycles}

In this section we prove Theorem~\ref{thm:general}. Let $H$ be a closed subgroup of $G$, and recall that $D_H(\mathfrak{g})= \{X- h\cdot X\mid h\in H\}$.

\begin{lem}
\label{lem:coinvts}
Suppose that $H= \overline{\langle h_1,\ldots, h_n\rangle}$ for some $h_1,\ldots, h_n\in H$. Then $D_H(\mathfrak{g})$ is the subspace spanned by $\{X- h_i\cdot X\mid X\in \mathfrak{g}, 1\leq i\leq n\}$.
\end{lem}

\begin{proof}
 Let $V$ be the subspace $\{X- h_i\cdot X\mid X\in \mathfrak{g}, 1\leq i\leq n\}$.  Let $m \in \N$.  Suppose $X- h\cdot X\in V$ for all $X\in \mathfrak{g}$ and all $h\in H$ such that $h$ is a word in the $h_i$ of length at most $m$.  Let $h$ be a word in the $h_i$ of length $m+1$ and let $X\in \mathfrak{g}$.  We can write $h= h_i h'$ for some $1\leq i\leq m$ and some $h'$ that is a word of length $m$ in the $h_i$.  Then $X- h\cdot X= X- h_ih'\cdot X= X- h'\cdot X+ h'\cdot X- h_ih'\cdot X= X- h'\cdot X+ h'\cdot X- h_i\cdot (h'\cdot X)\in V$.  It follows by induction on $m$ that $X- h\cdot X\in V$ for every $h\in \langle h_1,\ldots, h_n\rangle$. The result now follows.
\end{proof}

We recall some facts from the theory of 1-cocycles and representation varieties, see for instance \cite{Weil1964}. Let $\rho\in \Hom(\Gamma,G)$.  The group $\Gamma$ acts on $\mathfrak{g}$ by $\gamma\cdot X= \Ad_{\rho(\gamma)}(X)$; we denote the resulting $\Gamma$-module by $\mathfrak{g}(\rho)$.  Recall that a {\em 1-cocycle} is a function $\alpha\colon \Gamma\to \mathfrak{g}$ such that $\alpha(\gamma\gamma')= \alpha(\gamma)+ \gamma\cdot \alpha(\gamma')$.  We define $Z^1(\Gamma, \mathfrak{g}(\rho))$ to be the set of 1-cocycles; this is a vector space over $k$ with respect to the obvious pointwise operations.  There is a unique linear map $A_\rho$ from the tangent space $T_\rho(\Hom(\Gamma,G))$ to $Z^1(\Gamma, \mathfrak{g}(\rho))$ such that if $X= (X_1,\ldots, X_n)\in G^n$ then $A_\rho(X)(\gamma_i)= dR_{g_i}^{-1}(X_i)$ for $1\leq i\leq n$, where $g_i \deq \rho(\gamma_i)$ and $R_g\colon G\to G$ is multiplication by $g$.  Moreover, if $\gamma\in \Gamma$, $X\in T_\rho(\Hom(\Gamma,G))$ and $\alpha= A_\rho(X)$ then $d(\ev_\gamma)(X)= dR_g(\alpha(\gamma))$.

Now consider the special case $\Gamma= F_n= \langle \gamma_1,\ldots, \gamma_n\rangle$. Then for conjugacy classes $C_1,\ldots,C_n$, with $\C = C_1 \times \ldots \times C_n$ as before, the product map $f \colon C_1 \times \cdots \times C_n \to G$ coincides with $\ev_\gamma$, where $\gamma= \gamma_1\ldots \gamma_n$.  Let $c\in \C_0 = f^{-1}(\{1\})$.  Set $c_1'= 1$ and $c_i'= c_1\ldots c_{i-1}$ for $1\leq i\leq n$.  For any $X\in T_\rho(\Hom(\Gamma,G))$,
$$ df_\rho(X)= c_1'\cdot \alpha(\gamma_1)+\cdots  +c_n'\cdot \alpha(\gamma_n)= c_1'\cdot X_1+\cdots  +c_n'\cdot X_n. $$
Define $h\colon G^n\to \C$ by $h(g_1,\ldots, g_n)= (g_1c_1g_1^{-1},\ldots, g_nc_ng_n^{-1})$.  Let $(Y_1,\ldots, Y_n)\in \mathfrak{g}^n$.  It is easily seen that $dh_1(Y_1,\ldots, Y_n)= (dR_{c_1}(Y_1- c_1\cdot Y_1),\ldots, dR_{c_n}(Y_n- c_n\cdot Y_n))$.  We see that the derivative of $f\circ h$ is given by
\begin{eqnarray*}
 d(f\circ h) & = & c_1'\cdot (Y_1- c_1\cdot Y_1)+\cdots  +c_n'\cdot (Y_n- c_n\cdot Y_n) \\
  & = & (c_1'\cdot Y_1)- (c_1'c_1(c_1')^{-1})\cdot (c_1'\cdot Y_1)+\cdots +(c_n'\cdot Y_n)- (c_n'c_n(c_n')^{-1})\cdot (c_n'\cdot Y_n).
\end{eqnarray*}

\begin{proof}[Proof of Theorem~\ref{thm:general}]
 Let $V\subseteq \mathfrak{g}$ be the image of $df$.  Let $H$ be the subgroup of $G$ generated by $c_1,\ldots, c_n$.  Then $H$ is also the subgroup generated by $c_1'c_1(c_1')^{-1}, \ldots, c_n'c_n(c_n')^{-1}$.  It follows from Lemma~\ref{lem:coinvts} and the previous displayed equation that $V= D_H(\mathfrak{g})$.  By hypothesis, $\mathfrak{g}_H= 0$, so ${\rm im}(df)= D_H(\mathfrak{g})= \mathfrak{g}$.  This shows that $df$ has rank equal to $\dim(G)$.  Since $\rho$ is a smooth point of $\C$, it follows that $df$ has rank equal to $\dim(G)$ at all points in an open neighbourhood of $\rho$.  We deduce that ${\rm im}(f)$ is dense in $G$ and $f$ is smooth at $\rho$; in particular, any irreducible component of $\C_0$ that contains $\rho$ has dimension $\dim(\C)- \dim(G)$.
 
 Let $\D$ be an irreducible component of $\C_0$ that contains $G^{\circ}\cdot \rho$.  Then $\dim(G^{\circ}\cdot \rho)= \dim(G)$ as $\rho$ belongs to $(\C_0)_{\rm ir}$, so $\dim(\D)\geq \dim(G)$.  But $\dim(\D)= \dim(\C)- \dim(G)$ from the previous paragraph, so $\dim(\C)\geq 2\dim(G)$, with equality if and only if $\dim(G)= \dim(\D)$, which is the case if and only if $\rho$ is rigid. The result follows.
%
\end{proof}

\bibliographystyle{amsplain}
\bibliography{triangle-may-7}

\end{document}